\documentclass[a4paper,twoside, 10pt]{scrartcl}
\pdfoutput=1
\usepackage[english]{babel}
\usepackage[fixlanguage]{babelbib}
\usepackage{cite}
\usepackage{amssymb, amsmath, amstext, amsopn, amsthm, amscd, amsxtra, amsfonts}
\usepackage{yfonts, mathrsfs}
\usepackage{stackrel}
\usepackage{paralist}
\usepackage[all]{xy}
\usepackage[bookmarks,bookmarksnumbered,plainpages=false]{hyperref}
\usepackage[T1]{fontenc}
\usepackage[latin9]{inputenc}
\usepackage{colonequals}

\usepackage{setspace}

\swapnumbers
\usepackage{fancyhdr}

{%

\fancyhead[LE, RO]{{\bf \thepage}}
\fancyhead[RE]{{\bf\leftmark}}
\fancyhead[LO]{{\bf\rightmark}}
\cfoot[]{} \rfoot[]{}

\newtheoremstyle{standard}
 {16pt}  
 {16pt}  
 {}  
 {}  
 {\bfseries}
 {}  
 { } 
 {{{\thmnumber{#2.~}\thmname{#1}}}\thmnote{~(#3)}} 

\newtheoremstyle{kursiv}
 {16pt}  
 {16pt}  
 {\it}  
 {}  
 {\bfseries}
 {}  
 { } 
 {{{\thmnumber{#2.~}\thmname{#1}}}\thmnote{~(#3)}} 
	
\theoremstyle{standard}
\newtheorem{defn} [subsection]{Definition}

\theoremstyle{kursiv}
 
\newtheorem{thm}[subsection]{Theorem}

\newtheorem{cor} [subsection]{Corollary}
\newtheorem{lem} [subsection]{Lemma}

\numberwithin{equation}{section}

\newcommand{\aA}{\ensuremath{\mathcal{A}}}

\newcommand{\cC}{\ensuremath{\mathcal{C}}}

\newcommand{\sS}{\ensuremath{\mathcal{S}}}

\newcommand{\set}[1]{\{#1\}} 
\newcommand{\setm}[2]{\{\, #1\mid#2\,\}} 

\newcommand{\tl}{\textquotedblleft}
\newcommand{\tr}{\textquotedblright \ }

\newcommand{\op}{^{\rm op}}

\newcommand{\coloneq}{\mathrel{\mathop:}=}
\DeclareMathOperator{\Mod}{{\bf Mod}}

\DeclareMathOperator{\Ab}{{\bf Ab}}

\DeclareMathOperator{\Coker}{{Coker}}

\DeclareMathOperator{\Pres}{{\bf Pres}}

\DeclareMathOperator{\Hom}{Hom}

\DeclareMathOperator{\im}{Im}

\DeclareMathOperator{\Ob}{Ob}

\title{A construction of relatively pure submodules}
\author{Alexander Schmeding}
\date{}
\begin{document}
\thispagestyle{empty}
\maketitle
\abstract{In this paper we reconsider a classical theorem by Bican and El Bashir, which guarantees the existence of non-trivial relatively pure submodules in a module category over a ring with unit. Our aim is to generalize the theorem to module categories over rings with several objects. As an application we then consider the special case of $\alpha$-pure objects in such module categories.}

\section{Introduction}
In \cite{EBB2001} the flat cover conjecture for module categories has been solved positively, giving two independent approaches to prove the existence of flat covers. The first one due to Enochs uses certain properties of cotorsion theories. It seems to be better understood than the second one and has been used to generalize the existence theorems for flat covers to other categories; see for instance \cite{Aldrich2001}. The second method due to Bican and El Bashir is based on a complex set theoretic argument, where relatively pure submodules are constructed. This line of thought has been generalized to Grothendieck categories in \cite[Thm. 2.1]{bashir2}, where a different proof has been given. To our knowledge these argument have not been used very widely, although there are several applications as \cite{hk2010} and \cite{jh2008} show. The reluctance to use these arguments seems to be related to the fact that the proof given in \cite{EBB2001} is short and difficult to understand. We try to remedy this fact by giving a detailed explanation of the proof and several corollaries. The main theorem \ref{gen_thm} is a slight generalizations of the classical theorem by Bican and El Bashir. As our goal is to popularize the approach of Bican and El Bashir we show how to deduce this result from the classical theorem. In addition an interesting application will be discussed in section \ref{sect: appl}. However, 
the main results were already known (i.e. they appeared as \cite[Theorem 5]{EBB2001}, resp. \cite[Theorem 2.1]{CPT}). Nevertheless, the author thinks that detailed proofs are indispensable to understand these important theorems. Hence our aim is to remedy the lack of such proofs for these important results.

\paragraph{Acknowledgment.}
I would like to thank Henning Krause for stimulating discussions on the subject of this work.

\section{Projectively generated purities}
We fix an arbitrary small additive category $\cC$ and let $\sS$ be a skeleton for this category, i.e. a set of representatives of the objects $\Ob \cC$.

\begin{defn}
Denote by $\Mod \cC$ the category of $\cC$-modules, i.e. the additive functors $\cC\op \rightarrow \Ab$ with values in the abelian groups. Given $X \in \Mod \cC$, we define the \emph{cardinality of a module} $\lvert X\rvert \coloneq \sum_{C \in \sS} \lvert X(C)\rvert$.
\end{defn}

\begin{defn}[projectively generated purities]\label{defn: projgen}
$\aA$ be a class of  modules in $\Mod \cC$. We consider the class $\sigma_\aA$ of all short exact sequences in $\Mod \cC$, for which all $A \in \aA$ are projective, i.e. for every $\eta \in \sigma_\aA$ the induced sequence $\Hom_{\Mod \cC} (A , \eta )$ is exact in $\Ab$ for all $A \in \aA$. The class $\sigma_\aA$ will be called a \textit{purity projectively generated by $\aA$}.\\
A submodule $L\,\unlhd\, M$ of a module $M \in \Mod \cC$ is called \textit{$\sigma_\aA$-pure}, if the short exact sequence $\eta \colon 0 \rightarrow L \rightarrow M \rightarrow N \rightarrow 0$ is taken to an exact sequence in $\Ab$ by the functor $\Hom_{\Mod \cC} (A ,- )$ for every $A \in \aA$.  
\end{defn}
We remind the reader that the representable functors $\Hom (\cdot, A) , \ A \in \cC$ form a generating set of projective objects for $\Mod \cC$. Thus $\Mod \cC$ has enough projective objects. The next Lemma is preparatory in nature and will come in handy to prove the theorem of Bican and El Bashir:

\begin{lem}\label{lem: prep}
Let $\aA$ be some set of modules in $\Mod \cC$ and $N\, \unlhd\, M$ be in $\Mod \cC$. Fix for every $A \in \aA$ a short exact sequence $\rho_A \colon 0 \rightarrow K \stackrel{\iota}{\rightarrow} P \rightarrow A \rightarrow 0$ with $P$ some projective object. The submodule $N$ is $\sigma_\aA$-pure, if and only if for every pair of morphisms $s \colon K \rightarrow N$ and $s' \colon P \rightarrow M$ inducing a commutative diagram  
\begin{equation}\label{diag: sigpur}
\begin{split}
	\begin{xy}
	\xymatrix{
		K \ar[r]^{\iota} \ar[d]_{s} & P \ar[d]^{s'} \\     
		N \ar[r] & M
	}
	\end{xy}
\end{split}	
\end{equation}  
 there is a homomorphism $r \colon P \rightarrow N$, such that $r \iota = s$. 
\end{lem}

\begin{proof}
 Consider the commutative diagram 
\begin{displaymath}
	\begin{xy}
	\xymatrix{
		\rho_A \colon &    0 \ar[r] & K \ar[r]^a \ar[d]^s & P \ar[r]^b \ar[d]^{s'} & A \ar[r] \ar@{.>}[d]^{s''} & 0 \\
		\eta \colon   &	0 \ar[r]&  N \ar[r]^\iota & M \ar[r]^\pi & \Coker \iota \ar[r] & 0
	}
	\end{xy}
\end{displaymath} 
$s''$ is the canonical morphism induced by the cokernel. If $N$ is $\sigma_\aA$-pure, $\Hom (A, \eta)$ is an exact sequence in $\Ab$ and we derive a morphism $k \colon A \rightarrow M$, such that $\pi k = s''$. 
The required morphism exists by the homotopy Lemma (cf. \cite[Lemma B1]{model1994}. Conversely let $N$ be module with the above property and we wish to show that $\Hom_{\Mod \cC} (A , \eta)$ is exact. Since the functor is left exact it suffices to show that $\Hom (A, \pi)$ is an epimorphism. To this end let $f \colon A \rightarrow \Coker \iota$ be an arbitrary morphism. Now $P$ is projective, which implies that there is a morphism $k \colon P \rightarrow M$ with $\pi k = fa$ and since $N$ is the kernel of $\pi$, there is another morphism $k'\colon K \rightarrow N$ such that $\iota k' = k a$. The module $N$ has the above property with respect to the pair of morphisms $k',k$, there is $r \colon P \rightarrow N$ such that $ra = k'$. Again by the Homotopy Lemma, there is a morphism $w \colon A \rightarrow M$ such that $\Hom(A, \pi) (w) = f$ proving our claim.
\end{proof}

\section{Revisiting the approach of Bican and El Bashir to flat covers}

Throughout this section we are working in a module category $\Mod R$ over some (unital) ring $R$. We may think of $R$ as a category with a single object and therefore $\Mod R$ is in a canonical way a functor category over a small additive category with one object. Viewing $\Mod R$ as a special case of $\Mod \cC$ we may consider the objects defined in the last section in the setting of a module category. It is then easy to see that the notion of projectively generated purity coincides with the definition of Bican and El Bashirs paper (cf. \cite{EBB2001}). We may now formulate the main Theorem due to Bican and El Bashir: 

\begin{thm}\label{thm: bb1}
Let $R$ be an arbitrary (unital) ring and $\sigma$ a purity projectively generated by some set of modules $\aA$. For every cardinal $\lambda$ there is a cardinal $\kappa$, such that for every module $M$ in $\Mod R$ and $L\, \unlhd\, M$ with $\lvert M\rvert \geq \kappa$ and $\lvert M/L \rvert \leq \lambda$ the submodule $L$ contains another non trivial submodule, which is $\sigma$-pure in $M$.
\end{thm}


\begin{proof}[{Proof of Theorem \ref{thm: bb1}}]
In order to avoid problems later on, we may assume that the set $\aA$ is non-empty and contains at least one non trivial $R$-module, i.e. $\aA \neq \set{0}$. For every $A \in \aA$ fix a short exact sequence $0 \rightarrow U_A \rightarrow V_A \rightarrow A \rightarrow 0$ with a projective module $V_A$. Choose a set of elements $X_A$ which generates $U_A$ and a set $Y_A$ generating $V_A$ as an $R$-module. As a technical requirement we may assume without loss of generality, that $X_A$ and $Y_A$ are pairwise disjoint as sets for every $A$. We shall now outline the concept of the proof:

\paragraph{Leading idea:} A well known fact is that every submodule may be embedded into a pure submodule which may be constructed by adding suitable elements to the submodule (this process is called \tl purification\tr cf. \cite[Lemma 2.1.21]{prest2009}). We shall use the characterization of pure submodules as in Lemma \ref{lem: prep} to pursue exactly this line of thought. Since we fixed a set of generating elements $Y_A$ for every $V_A$ and $A \in \aA$, Lemma \ref{lem: prep} shows that every element which has to be added to an arbitrary (non-pure) submodule is contained in a homomorphic image of one of the $Y_{A}$. To use these observation we will at first introduce a way to count the homomorphic images of the generating sets of $V_A$ which we need to add to a submodule. It will turn out that there is an upper bound for the number of choices necessary to purify an arbitrary submodule. In a second step we construct cardinals which satisfy the properties claimed in the theorem. 
    
\paragraph{1st step:} In order to count the number of distinct choices for elements which are joined to a (non-pure) submodule we want to define so called \textit{formal solutions}. One may think of a formal solution as a counting variable which uniquely identifies each element in the purification process of a submodule but does not depend on the particular submodule.\\
Let $\aleph$ be an infinite regular cardinal such that $\aleph > \sup \set{\lvert R\rvert , \lvert \aA\rvert , \sup \setm{\lvert X_A\rvert}{A \in \aA}}$. We begin the construction of formal solutions by defining the set $x_0 \coloneq \set{\emptyset}$. The set $x_0$ is chosen just for technical reasons and it is not important what $x_0$ really is (in fact Bican and El Bashir just define $x_0$ to be a \tl new symbol\tr ). The construction of formal solutions is inductively given by the following statements: 
\begin{compactitem}
\item[1.]Let $x_0$ be the set just defined, then we declare that $\set{x_0}$ is a \textit{formal solution}.
\item[2.]We introduce the following notation: For any set of formal solutions $S$ we denote the set $\bigcup_{y \in S} y$ by $\lVert S\rVert $. If $S$ is a set of formal solutions with $\lvert S\rvert \leq \aleph$ and for any $A \in \aA$ there is a map $h\colon X_A \rightarrow \bigoplus_{y \in \lVert S\rVert } R$, then $z = Y_A \times \set{h}$ is a formal solution. 
\end{compactitem} 
Notice that the first statement just gives us a starting point for the inductive definition, and we have assured that there is a distinguished formal solution, which is given as the set containing $x_0$. The formal solutions constructed by inductive application of step 2 should be seen as a way of counting homomorphic images of $X_A$ via the coefficients for every element in $X_A$ together with the generating set $Y_A$. We now have to prove that the class of all formal solutions is a set. This follows from the fact, that for every $A \in \aA$ there are at most $(\aleph \cdot \lvert R\rvert)^{\lvert X_A\rvert} \leq \aleph$ maps from $X_A$ into direct sums of $R$ with less than $\aleph$ summands. For the estimate we used $\aleph > \lvert R\rvert , \lvert X_A \rvert$ and \cite[Thm. 5.20 (iii) (a)]{jechst}, which is applicable since $\aleph$ is a regular cardinal with $\sup_{A \in \aA} \lvert X_A\rvert < \aleph = \text{cf } \aleph$ Therefore there are at most $\sum_{A \in \aA} (\aleph \cdot \lvert R\rvert)^{\lvert X_A\rvert} \leq \lvert \aA \rvert \cdot \aleph = \aleph$ formal solutions which may be generated by step 2. The class of all formal solutions is a set whose cardinality we denote by $\Xi$. As already seen, $\Xi \leq \aleph$ holds. By step one there is at least one formal solution. Combining the assumption that there is at least one non-trivial module in $\aA$ clearly by the inductive application of step 2 there are infinitely many formal solutions and thus $\Xi \geq \aleph_0$.\\
 A (non-empty) set of formal solutions $\chi $ is called \textit{height} if, for any 
	\begin{displaymath}
	y = Y_A \times \{ h\} \in \chi ,\quad h\colon X_A \rightarrow \bigoplus_{\lVert S\rVert } R
	\end{displaymath} 
the condition $S \subseteq \chi$ holds. By definition every height contains the formal solution $\set{x_0}$. Consider a height $\chi$ and some module $M \in \Mod R$ together with a map $f\colon \lVert \chi \rVert \rightarrow M $. The map $f$ is called \textit{realization of $\chi$ in $M$ over $a$}, if the following conditions are satisfied.\newpage 
\begin{compactitem}
\item[1.] $f(x_0) = a$.
\item[2.] for all $y = Y_A \times \set{h} \in \chi , \ h\colon X_A \rightarrow \bigoplus_{\lVert S\rVert } R$ there is a homomorphism $g_A \colon V_A \rightarrow M$, which makes the following diagrams commutative:
\begin{equation}\label{natmor}
\begin{split}
	\begin{xy}
	\xymatrix{
		y =Y_A \times \set{h} \ar[r] \ar[d]^i \ar@{}[dr] |{\rm i)}& \lVert\chi\rVert \ar[d]^f & & X_A \ar[d]^\subseteq \ar[r]^{\hspace{-0.5cm} h} \ar @{} [dr] |{\rm ii)}& \bigoplus_{\lVert S \rVert } R \ar[d]^{\overline{f}} \\
		V_A \ar[r]^{g_A} & M & & U_A \ar[r]^{g_{A_{\mid U_A}}} & M
}
	\end{xy}
\end{split}	
\end{equation} 
Here $i\colon Y_A \times \set{h} \rightarrow V_A,$ is given by $y_A \times \set{h} \mapsto y_A$ and the map $\overline{f}$ is derived from the unique extension property of free modules as 
$\overline{f} (r_y)_{y \in \lVert S \rVert} \coloneq \sum_{y \in \lVert S\rVert} f(y)\cdot r_y$.
\end{compactitem}
We have now assembled all tools necessary to prove the theorem. 
\paragraph{2nd step:} Consider a cardinal $\lambda \geq 2$ and choose a cardinal $\mu \geq \Xi$, with $\lambda^\mu \geq \sup \setm{\lvert Y_A\rvert}{A \in \aA}$.\footnote{Observe that the prove diverges here in a subtle way from the original proof by Bican and El Bashir. In their paper $\mu = \Xi$ holds, but the author does not know how to obtain the estimates necessary later on. However for the purity generated by the finitely presented modules (i.e. the classical case) indeed $\mu = \Xi$ suffices as $\sup \setm{\lvert Y_A\rvert}{A \in \text{Pres}_{\aleph_0} (R)} = \aleph_0 \leq \Xi$.} We define $\kappa = (\lambda^\mu )^{+}$, where $\alpha^{+}$ is the cardinal successor of a cardinal $\alpha$. Let $L\,\unlhd\, M$ be a submodule of $M$, such that $\lvert M\rvert \geq \kappa$ and $\lvert M/L \rvert \leq \lambda$. We give an indirect proof by assuming that there is no non-trivial submodule of $L$ which is $\sigma_\aA$-pure in $M$.\\
To obtain a contradiction we will iteratively construct partitions of $\kappa$. In the end we will show that our assumption implies the existence of $\Xi^+$ formal solutions which is the desired contradiction.
Choose pairwise distinct elements $a_\alpha \in M,\ \alpha < \kappa$. For every $\alpha < \kappa$ define the height $\chi_{\alpha} \coloneq \set{\set{x_0}}$ together with the realization $f_\alpha \colon \lVert \chi_\alpha \rVert \rightarrow M,\ f_\alpha (x_0) = a_\alpha$.\\
Let $\pi \colon M \rightarrow M/L$ be the canonical projection onto the quotient and define a relation \tl $\tau$\tr on $\kappa$ via: $\alpha \tau \beta$ if and only if $\pi f_\alpha = \pi f_\beta$. Clearly $\tau$ is an equivalence relation and we denote by $P(\tau )$ the corresponding  partition of $\kappa$. For an arbitrary $T \in P(\tau )$, let $\min T$ be the smallest ordinal contained in $T$. For any $\beta \in T$ with $\beta > \min T$ the construction of $T$ yields $a_\beta - a_{\min T} \in L$. Thus we obtain a map 
	\begin{displaymath}0 \neq f_\beta - f_{\min T} \colon \lVert \chi_\beta \rVert \rightarrow L.\end{displaymath}
We claim that there exists a height $\chi_{0\beta}$ and a realization $f_{0\beta}$ of $\chi_{0\beta}$ in $M$ over $a_\beta - a_{\min T}$, which properly extends $f_\beta - f_{\min T}$. Consider the following construction: \\
The submodule $M_\beta \coloneq (a_\beta - a_{\min T})R$ is not $\sigma_\aA$-pure in $M$ by our assumption. Thus by Lemma \ref{lem: prep} there exists $A \in \aA$ and morphisms $s\colon U_A \rightarrow M_\beta$ and $s'\colon V_A \rightarrow M$ as in \eqref{diag: sigpur}, such that $s$ does not factor through $\iota$ (from \eqref{diag: sigpur}). We have to construct a height $\chi_{0\beta}$ with $\chi_\beta \subsetneq \chi_{0\beta}$. To this end we wish to construct $y \coloneq Y_A \times \set{\theta}$, where $A$ is the $A$ from above and $\theta$ is a suitable map. We postpone the problem of finding $\theta$ and define first $\chi_{0\beta} \coloneq \set{\set{x_0}, y}$ and a map $f_{0\beta}\colon \lVert \chi_{0\beta} \rVert \rightarrow M$ such that $f_{0\beta \mid\chi_\beta} = f_\beta - f_{\min T}$ and $f_{0\beta} (y_A \times \set{\theta} ) \coloneq s' (y_A)$ hold. By construction the diagram \eqref{natmor} i) with $g_A = s'$ is commutative. For each $x_A \in X_A$ we may choose and fix some (probably non unique) $r_{x_A} \in R$ satisfying $s(x_A) = (a_\beta - a_{\min T})\cdot r_{x_A}$. Now define 
	\begin{displaymath}
	\theta \colon X_A \rightarrow \oplus_{\set{x_0}} R =R \text{ as } \theta(x_A) = r_{x_A} \text{, if }s (x_A)= (a_\beta - a_{\min T})\cdot r_{x_A}.
	\end{displaymath}
 This choice of $\theta$ turns $\chi_{0\beta}$ into a height and the diagram \eqref{natmor} ii) commutes. The realization $f_{0\beta}$ properly extends $f_\beta - f_{\min T}$: If it did not properly extend $f_\beta - f_{\min T}$, one obtains $\im f_{0\beta} \subseteq M_\beta$ and by commutativity of \eqref{natmor} i) we deduce $\im s' \subseteq M_\beta$. The choice of morphisms $s,s'$ implies $s = s' \iota$ (by \eqref{diag: sigpur}). Thus $s$ extends to a morphism $V_A \rightarrow M_\beta$, which contradicts our choice of $s$ and $s'$. 
\paragraph{}
Continuing our construction we consider $\kappa_0 = \kappa \setminus \setm{\min T}{T \in P(\tau )}$ and introduce an equivalence relation \tl $\tau_0$\tr with corresponding partition $P(\tau_0)$ on $\kappa_0$ in the following way: 
\begin{displaymath}
\alpha \tau_0 \beta :\Longleftrightarrow  \alpha \tau \beta ,\ \chi_{0\alpha} = \chi_{0\beta} \text{ and } \pi f_{0\alpha} = \pi f_{0\beta} 
\end{displaymath}
We need an estimate for the number of equivalence classes in $P(\tau_0)$: The requirement $\lvert M/L\rvert \leq \lambda$ yields at most $\lambda$ equivalence classes in $P(\tau)$. Every height $\chi_{0\beta}$ contains exactly two formal solutions, one being $\set{x_0}$ by construction. The set of all formal solutions contains $\Xi \leq \mu$ elements. Hence there are at most  $\lambda^\mu$ distinct choices for classes, which satisfy the first two requirements of $\tau_0$. Finally we have to compute how many choices are added for each class by the third condition of $\tau_0$. To this end choose realizations $f_{0\alpha}, f_{0\beta}$ such that $\alpha \tau \beta$ and $\chi_{0\alpha} = \chi_{0\beta} = \set{\set{x_0}, Y_A \times \set{\theta}}$ holds. 
For each $z \in Y_A$ there are $\lambda$ choices for $\pi f_{0\alpha} (z)$, resp. $\pi f_{0\beta} (z)$. By choice of $\mu$, 
there are at most $\lambda^\mu$ different choices to satisfy the third requirement. 
Thus $P(\tau_0) = \lambda^\mu \cdot \lambda^\mu = \lambda^\mu$ holds. The complement $\setm{\min T}{T \in P(\tau )}$ of $\kappa_0$ in $\kappa$ contains at most $\lambda$ different elements, since $\alpha \tau \beta$ iff $\pi f_\alpha = \pi f_\beta$ and $\lvert M/L\rvert \leq \lambda$. For later use we remark, that the complement is non-empty and contains at most $\lambda^{\mu}$ elements.\\
We shall continue the construction of  partitions via transfinite induction with $\Xi^+$ steps. Choose the equivalence relation $\tau_0$ as start of the induction and observe that isolated and limit steps can be formulated in the same way. \\
Let $0 < \gamma < \Xi^+$ and assume that for all $\delta < \gamma$ the sets $\kappa_\delta$, the heights $\chi_{\delta \beta}$ and their realizations $f_{\delta \beta}$ for $\beta \in \kappa_\delta$ as well as the equivalences $\tau_\delta$ on $\kappa_\delta$ have already been constructed, such that the following is satisfied:
\begin{compactitem}
\item[1.] $\kappa_\delta = \bigcap_{\delta_1 < \delta} \kappa_{\delta_1} \setminus \setm{\min T}{T \in P (\bigcap_{\delta_1 < \delta} \tau_{\delta_1})}$.
\item[2.] The equivalence $\tau_\delta$ on $\kappa_\delta$ is  defined by $\alpha \tau_\delta \beta :\Longleftrightarrow \alpha \tau_{\delta_1} \beta \ \forall \delta_1 < \delta$, $\chi_{\delta \alpha} = \chi_{\delta \beta}$ and $\pi f_{\delta \alpha} = \pi f_{\delta \beta}$.
\item[3.] $\bigcup_{\delta_1 < \delta} \chi_{\delta_1 \beta} \subsetneq \chi_{\delta \beta}$.
\item[4.] $f_{\delta \beta}$ properly extends $f_{\delta_1 \beta} - f_{\delta_1 \min T}$ for all $\delta_1 < \delta$, $\beta \in \kappa_\delta$ and $\beta \in T \in P (\bigcap_{\delta_1 < \delta} \tau_{\delta_1})$,
\end{compactitem}
Since $\mu \geq \Xi \geq \aleph_0$ and $\lambda >1$, standard cardinal arithmetic shows $\mu + \mu = \mu \cdot \mu = \mu,\ \mu < \lambda^\mu$ and $(\lambda^\mu )^\mu = \lambda^\mu$ (cf. \cite[Thm. 5.20 (iii)]{jechst}). We may now extend every height with at most $\Xi$ distinct elements and the definition of the partitions for $\gamma < \Xi^+ \leq \mu^+$ allows the estimate:
\begin{displaymath}
 \biggl\lvert  P(\bigcap_{\gamma_1 < \gamma} \tau_{\gamma_1}) \biggm\rvert \leq \biggm\lvert \prod_{\gamma_1 < \gamma} P(\tau_{\gamma_1}) \biggr\rvert  \leq (\lambda^\mu)^\mu = \lambda^\mu.
\end{displaymath}
Set $\kappa_\gamma \coloneq \bigcap_{\delta < \gamma} \kappa_\delta \setminus \setm{\min T}{T \in P(\bigcap_{\delta_1 < \gamma} \tau_{\delta})}$. Using the above estimates we compute 
	\begin{align*}
	 \lvert \kappa \setminus \kappa_\gamma \rvert &= \left\lvert \left(\bigcup_{\delta < \gamma} \kappa \setminus \kappa_\delta \right) \cup \setm{\min T}{T \in P(\bigcap_{\delta < \gamma} P(\tau_{\delta}))} \right\rvert \\ 
	  &\leq \left \lvert \left(\bigcup_{\delta < \gamma} \kappa \setminus \kappa_\delta \right) \right\rvert + \left\lvert P(\bigcap_{\delta < \gamma} P(\tau_{\delta})) \right\rvert \leq  \Xi \cdot \lambda^\mu + \lambda^\mu \leq \mu \lambda^{\mu} + \lambda^\mu =\lambda^\mu.
	\end{align*}
In particular the estimate implies, that $\kappa_\gamma$ is non-empty, since otherwise $\lvert \kappa \setminus \kappa_\gamma\rvert = \kappa = (\lambda^\mu)^+$ holds, which contradicts our last estimate. 
Let $\beta \in T \in P(\bigcap_{\delta < \gamma} \tau_\delta )$ with $\beta > \min T$. Choose classes $T_\delta \in P(\bigcap_{\delta_1 < \delta} \tau_{\delta_1}) ,\ \delta < \gamma $ such that $T \subseteq \bigcap_{\delta < \gamma} T_\delta$.
For each $\delta_1 < \delta < \gamma$ we have $\chi_{\delta_1 \beta} = \chi_{\delta_1 \min T}$ by 2. and by using 4. the map $f_{\delta \beta}$ properly extends $f_{\delta_1 \beta} - f_{\delta_1 \min T_\delta}$, while $f_{\delta \min T}$ properly extends $f_{\delta_1 \min T} - f_{\delta_1 \min T_\delta}$.\\
Thus $f_{\delta \beta} - f_{\delta \min T} \colon \lVert \chi_{\delta  \beta}\rVert \rightarrow L$ properly extends
\begin{displaymath}
(f_{\delta_1 \beta} - f_{\delta_1 \min T_\delta})-(f_{\delta_1 \min T} - f_{\delta_1 \min T_\delta})=f_{\delta_1 \beta} - f_{\delta_1 \min T} \colon \lVert \chi_{\delta_1 \beta}\rVert \rightarrow L.
\end{displaymath} 
 In particular this shows that   
\begin{displaymath}
f \coloneq \bigcup_{\delta < \gamma} (f_{\delta \beta} - f_{\delta \min T}) \colon \left\lVert \bigcup_{\delta < \gamma} \chi_{\delta \beta}\right\rVert \rightarrow L 
\end{displaymath}
is a well-defined map and a realization of the height $\bigcup_{\delta < \gamma} \chi_{\delta \beta}$. The map $f$ is a non trivial morphism since $f(x_0) = a_\beta - a_{\min T} \neq 0$. We conclude that the submodule $0 \neq \langle \im f \rangle \,\unlhd\, L$, which is generated by $\im f$ can not be $\sigma_\aA$-pure in $M$ by our assumption.\\
 Now we wish to construct a height together with a realization which properly extends the maps already constructed. Since the submodule is not $\sigma_\aA$-pure, we may again choose some $A \in \aA$ and a pair of morphisms $s \colon U_A \rightarrow \langle \im f\rangle$, $s'  \colon V_A \rightarrow M$, such that $s$ does not factor through the inclusion $U_A \rightarrow V_A$. We shall now argue as at the start of the induction, claiming that the argument is valid for arbitrary $0 < \gamma < \Xi^+$. Indeed we have only used the properties $\im s\, \unlhd\, \langle \im f \rangle$ and $\lvert S\rvert \leq \aleph$. For arbitrary $0 < \gamma < \Xi^+$ these properties are clearly satisfied, as $\Xi \leq \aleph$ holds. It remains to show that there is a map $h$, such that the diagram \eqref{natmor} ii) commutes (completing the construction of the height).\\
We remark that since $\lvert X_A \rvert \leq \aleph$, for all $A \in \aA$ we may clearly choose $h$, such that $\lvert S\rvert \leq \aleph$ and the diagram \eqref{natmor} ii) commutes for this choice of $h$ and $g_A \coloneq s$. Now we need to assure that the formal solution $X_A \times \set{h}$ induced by $h$ is not contained in $\bigcup_{\delta < \gamma} \chi_{\delta \beta}$. Assuming the contrary, the definition of a height implies that there is a commutative diagram of type \eqref{natmor} ii) for some $g_A$. We derive $g_{A\mid U_A} = f h = f_{\gamma \beta}  h = s$ and thus $s$ factors through $U_A\rightarrow V_A$ contradicting our choice of $s$.\\ 
We have constructed a height $\chi_{\gamma \beta}$ with $\bigcup_{\delta < \gamma} \chi_{\delta \beta} \subsetneq \chi_{\gamma \beta}$ together with a realization $f_{\gamma \beta}$ of $\chi_{\gamma \beta}$. The constructed map properly extends $f_{\delta \beta} - f_{\delta \min T}$ for all $\delta < \gamma$.\\
Define another equivalence relation $\tau_\gamma$ on $\kappa_\gamma$ via 
\begin{displaymath}
 \alpha \tau_\gamma \beta :\Longleftrightarrow \alpha \tau_\delta \beta ,\ \text{for all } \delta < \gamma \text{ one has } \chi_{\gamma \alpha} = \chi_{\gamma \beta} \mbox{ and } \pi f_{\gamma \alpha} = f_{\gamma \beta}.
\end{displaymath}
The cardinality of the complement of $Z \coloneq \bigcap_{\gamma < \Xi^+} \kappa_\gamma $ in $\kappa$ is bounded by $\Xi^+ \cdot \lambda^\mu = \lambda^\mu$. Analogous to the induction step, one may argue that $Z$ is non-empty: If it were empty, its complement $\kappa \setminus Z$ would be all of $\kappa$, hence $\lambda^\mu \geq \lvert \kappa \setminus Z \rvert = \kappa = (\lambda^\mu)^+$ which is absurd. The heights $\chi_{\gamma \beta}$ with $\beta \in \bigcap_{\gamma < \Xi^+} \kappa_\gamma $ and $\gamma < \Xi^+$ form an ascending chain of length $\Xi^+$. By construction $\chi_{\delta \beta} \subsetneq \chi_{\gamma \beta}, \ \forall \delta < \gamma$. Therefore there have to be at least $\Xi^+$ formal solutions, contradicting our choice of $\Xi$. 
\end{proof}

The idea used in the proof of Theorem \ref{thm: bb1}, is, roughly speaking, to purify a submodule by adding suitable elements and counting the necessary steps. In the next section we will discuss a slight generalization of the previous theorem to functor categories.    

\section{Purities for rings with several objects}

In this section we study the notion of purity in the category of additive functors over an additive category. We shall use functor rings which were introduced by Gabriel in \cite{Gabriel62}. Throughout this section $\cC$ will again denote a small additive category.

\begin{defn}
Define the functor ring associated to the category $\cC$ with skeleton $\sS$: 
	\begin{displaymath}
		R_\cC \coloneq \bigoplus_{A,B \in \sS} \Hom_{\cC} (A,B)
	\end{displaymath}
Recall from \cite[Chapitre II]{Gabriel62} that $R_\cC$ is a ring with enough orthogonal idempotents (but possibly without a unit), and that $R_\cC$-modules are defined in the usual way adding the condition $M.R_\cC = M$.
\end{defn}

It is a classical result by Gabriel (see \cite[Chapitre II, Proposition 2]{Gabriel62}) that there is an equivalence of categories $S \colon \Mod \cC \rightarrow \Mod R_\cC$, which preserves the cardinality:
	\begin{displaymath}
		\lvert X\rvert = \lvert S(X)\rvert , \forall X \in \Mod \cC,
	\end{displaymath}
 where the cardinality on the right hand side is the cardinality of the underlying set. Using this equivalence it is easy to prove a generalization of theorem \ref{thm: bb1}. The theorem appeared in this generality first without proof in \cite[Theorem 2.1]{CPT}, where it is used to generalized the flat cover theorem.  	

\begin{thm}\label{gen_thm}
Let $\cC$ be an arbitrary small additive category together with a purity $\sigma$ projectively generated by a set $\aA \subseteq \Ob \Mod \cC$. For every cardinal $\lambda$ there is a cardinal $\kappa$, such that for every module $M \in \Mod \cC$ and $L\, \unlhd\, M$ with $\lvert M\rvert \geq \kappa$ and $\lvert M/L \rvert \leq \lambda$ the submodule $L$ contains a non-trivial submodule which is $\sigma$-pure in $M$.
\end{thm}

\begin{proof}
%
By Lemma \ref{lem: prep} being a $\sigma$-pure subobject is a categorical property. We denote by $S(\aA)$ the set of objects $\setm{S(A)}{A \in \aA}$ and let $\sigma_{S(\aA)}$ be the purity generated by $S(\aA)$ in $\Mod (R_\cC)$. In this category Lemma \ref{lem: prep} holds and since $S$ is an equivalence, $L$ will contain a non-trivial $\sigma$-pure subobject of $M$ if and only if $S(L)$ contains a non-trivial $\sigma_{S(\aA)}$-pure submodule of $S(M)$. Now $\lvert L\rvert =\lvert S(L)\rvert $ and $\lvert M\rvert =\lvert S(M)\rvert$ holds, and we just need to prove theorem \ref{thm: bb1} for non-unital rings with enough orthogonal idempotents:
\\ Reconsidering step one of \ref{thm: bb1}, clearly our definitions did not depend on the the ring to be unital. Since $\Mod R_\cC$ also has enough projectives there is no need to change anything in this step.\\
Continuing with the second step, the proof uses on several occasions properties of modules over a ring with a unit. There is just one issue which has to be addressed:\\
We have to assure that the submodules constructed in the second step are non trivial, resp. that the construction yields larger submodules. Since for every (non trivial) element there is a local unit, it is straight forward to check that this is indeed the case. 
Thus the second step of theorem \ref{thm: bb1} may be carried out without any changes. The other arguments used in the second step are purely set theoretic in nature and thus hold without any changes.\\
In conclusion our investigation shows that theorem \ref{thm: bb1} also holds for the (non unital) ring with enough idempotents $R_\cC$.
\end{proof}

We shall now apply theorem \ref{gen_thm} to the concept of $\alpha$-pure morphisms. 

\section{Application: \texorpdfstring{$\alpha$}{alpha}-pure subobjects} \label{sect: appl}

Recall that an additive category is \emph{locally presentable} if it is cocomplete and admits a a generating set of objects that are $\alpha$-presentable for some regular cardinal $\alpha$ (cf. \cite{fp1994}). An object $X$ is called \emph{$\alpha$-presentable} if the representable functor $\Hom (X,-)$ preserves $\alpha$-filtered colimits. It is a well known fact that the category of modules $\Mod \cC$ for every small additive category $\cC$ is locally presentable (since the representable functors form a set of $\aleph_0$-presentable objects)

\begin{defn}
Let $\alpha$ be some infinite regular cardinal. A morphism $f \colon X \rightarrow Y$ in $\Mod \cC$ is called
	\begin{compactenum}
	\item  \emph{$\alpha$-pure monomorphism} if it is the $\alpha$-filtered colimit of split monomorphisms. 
	\item  \emph{$\alpha$-pure quotient}  if it is the $\alpha$-filtered colimit of split epimorphisms.  
	\end{compactenum}
If $X \rightarrow Y$ is an $\alpha$-pure monomorphism, we call $X$ \emph{$\alpha$-pure subobject} of $Y$. 	
\end{defn}

Our notion of $\alpha$-pure subobjects and $\alpha$-pure quotients coincides with the usual definitions of such objects (cf. \cite{pureqout}) in $\alpha$-accessible categories. The category $\Mod \cC$ is a $\aleph_0$-locally presentable as we already remarked. Thus by \cite[Remark after Thm. 1.20]{fp1994} it is a $\lambda$-locally presentable category for each regular infinite cardinal $\lambda \geq \aleph_0$. Locally $\alpha$-presentable categories are $\alpha$-accessible. In $\Mod \cC$ our definition coincides with the usual definition of $\alpha$-pure objects and we may thus freely use the results from \cite{pureqout} on pure morphisms.\\
Fix an arbitrary infinite regular cardinal $\alpha$. The category $\Mod \cC$ is locally $\alpha$-presentable. Thus there is a set $\Pres_\alpha$ of representatives of all $\alpha$-presentable objects in $\Mod \cC$. We consider the purity $\sigma_\alpha$ projectively generated by $\Pres_\alpha$. 

\begin{lem}
For every infinite regular cardinal $\alpha$, the class of all $\sigma_\alpha$-pure subobjects coincides with the class of all $\alpha$-pure subobjects.
\end{lem}

\begin{proof}
The category $\Mod \cC$ is an abelian $\alpha$-accessible category thus by \cite[Proposition 5]{pureqout}, $\alpha$-pure subobjects are precisely the kernels of $\alpha$-pure quotients and $\alpha$-pure quotients are precisely the cokernels of $\alpha$-pure monomorphisms. We will use this knowledge to prove our claim. 

\paragraph{}Let $L \unlhd M$ be an arbitrary $\sigma_\alpha$-pure subobject and consider the short exact sequence 
	\begin{displaymath}
		\eta \colon 0  \rightarrow L \stackrel{\iota}{\rightarrow} M \stackrel{\pi}{\rightarrow} C \rightarrow 0
	\end{displaymath}
 where $C$ is the cokernel of the inclusion. Since $L$ is $\sigma_\alpha$-pure, for every $A \in \Pres_\alpha$, $\Hom_{\Mod \cC} (A, \eta)$ is a short exact sequence in $\Ab$. In other words, the cokernel $C$ is projective with respect to every $A \in \Pres_\alpha$ (cf. \cite[Definition 1]{pureqout}). An application of \cite[Proposition 3]{pureqout} now implies that $\pi\colon M \rightarrow C$ is an $\alpha$-pure quotient. Since $\iota$ is the kernel of $\pi$, our previous observations imply that $L$ is an $\alpha$-pure subobject.\\
Conversly let $\iota \colon L \rightarrow M$ be an $\alpha$-pure monomorphism, then its cokernel $\pi \colon M \rightarrow C$ is a $\alpha$-pure quotient. Again by \cite[Proposition 3]{pureqout} $\alpha$-pure quotients are projective with respect to every object in $\Pres_\alpha$ . Thus considering again the short exact sequence $\eta$ from above, $\Hom_{\Mod \cC} (A, \eta)$ is an exact sequence for every $A \in \Pres_\alpha$. By definition $L$ must be a $\sigma_\alpha$-pure submodule.  
 \end{proof}
 
Clearly Theorem \ref{gen_thm} is applicable to $\sigma_\alpha$. Combining the theorem and our last lemma we derive the following corollary

\begin{cor}
Let $\alpha$ be an arbitrary infinite regular cardinal and $\lambda$ be an arbitrary cardinal. Then there is a cardinal $\kappa$, such that for every $M \in \Ob \Mod \cC$ and $L\, \unlhd\, M$ with $\lvert M\rvert \geq \kappa$ and $\lvert M/L \rvert \leq \lambda$ the submodule $L$ contains a non-trivial $\alpha$-pure subobject of $M$.
\end{cor}

\bibliographystyle{alpha}
\newpage
\phantomsection{}
\addcontentsline{toc}{section}{References}
\bibliography{Notemain}

\end{document}